\documentclass[11pt,a4paper]{amsart}
\usepackage[utf8]{inputenc}
\usepackage{amsmath,amsthm,enumitem}
\newtheorem{theorem}{Theorem}
\usepackage{amsfonts}
\usepackage{enumitem}

\usepackage{fmtcount}
\newtheorem{lm}[theorem]{Lemma}

\newcommand{\abs}[1]{\lvert #1 \rvert}

\def\E{\mathbb{E}}
\newcommand{\Hh}{{\mathcal H}}
\newcommand{\D}{{\mathcal D}}
\newcommand{\C}{{\mathcal C}}
\newcommand{\A}{{\mathcal A}}
\newcommand{\B}{{\mathcal B}}

\def\Pr{\mathbb{P}}
\def\Var{\mathrm{Var}}
\def\diam{\mathrm{diam}}

\def\Dec{\mathsf{Dec}}
\def\Z{\mathbb{Z}}
\def\N{\mathbb{N}}

\def\Pp{\mathcal{P}}

\usepackage{amssymb}
\theoremstyle{definition}

\usepackage{hyperref}
\title{Metric decomposability theorems on sets of integers}
\author{Pierre-Yves Bienvenu}
\address{Institut für Analysis und Zahlentheorie\\TU Graz\\Kopernikusgasse 24\\8010 Graz Austria}
\email{bienvenu@tugraz.at}
\begin{document}

\begin{abstract}
A set $\mathcal{A}\subset \mathbb{N}$ is called additively decomposable
(resp. asymptotically additively decomposable) if there exist sets $\mathcal{B},\mathcal{C}\subset \mathbb{N}$ of cardinality at least two each
such that $\mathcal{A}=\mathcal{B}+\mathcal{C}$ (resp. $\mathcal{A}\Delta (\mathcal{B}+\mathcal{C})$ is finite). If none of these properties hold, the set $\A$ is called totally primitive.
We define $\mathbb{Z}$-decomposability analogously with subsets $\mathcal{A,B,C}$ of $\Z$.
Wirsing showed that almost all subsets of $\mathbb{N}$ are totally primitive.
In this paper, in the spirit of Wirsing, we study decomposability from
a probabilistic viewpoint. First, we show that almost all symmetric subsets of $\mathbb{Z}$ are $\mathbb{Z}$-decomposable. Then
we show that almost all small perturbations of the set of primes yield a totally primitive
set. Further, this last result still holds when the set of primes is replaced by the set of sums of two squares, which is by definition decomposable.
\end{abstract}
\maketitle
\begin{flushright}
\textit{In memoriam Eduard Wirsing.}
\end{flushright}

\section{Statement of the results}
This article is concerned with sum sets in the integers.
Given two subsets $\mathcal{A},\mathcal{B}$ of $\Z$, their sum set $\mathcal{A}+\mathcal{B}$ is the set $\{a+b : (a,b)\in \mathcal{A}\times \mathcal{B}\}$.
We denote $2\A=\A+\A$.
Further, we denote by $\mathcal{A}\sim \mathcal{B}$ the property that $\mathcal{A}\Delta \mathcal{B}$ is finite.
A set $\mathcal{C}\subset \N$ is called \textit{additively decomposable}
(resp. \textit{asymptotically additively decomposable}) if there exist sets $\mathcal{A},\mathcal{B}\subset\N$
 of cardinality at least two each
 such that
(resp. $\mathcal{C}\sim \mathcal{A}+\mathcal{B}$). 
A set which is not asymptotically additively decomposable is called
totally primitive.
Similarly, a set $\mathcal{C}\subset \Z$ is called \textit{additively $\Z$-decomposable} (resp \textit{asymptotically additively $\Z$-decomposable}) if there exist sets $\mathcal{A},\mathcal{B}\subset\Z$ of cardinality at least two each such that $\mathcal{C}= \mathcal{A}+\mathcal{B}$ (resp.
$\mathcal{C}\sim \mathcal{A}+\mathcal{B}$). 
%Since this paper will discuss no other notion of decomposability, we
%will abandon the adverbs and simply call such a set \textit{decomposable} or $\Z$-\textit{decomposable} respectively.
%Otherwise it is called indecomposable or $\Z$-indecomposable.

An old conjecture of Ostmann \cite[page 13]{ostmann} asserts that the set $\Pp$ of primes is totally primitive. In spite of serious efforts by numerous authors and
notable advances (see \cite{els,elsHarp} and the very recent \cite{gyori} on related problems), the problem remains unsolved.
The philosophy supporting this idea is that additive decomposability is a very rare property, so most sets occurring in number theory which are not specifically defined to be a sum set\footnote{However, examples such as the set of sums of two squares, which can be defined through a multiplicative property ($n$ is a sum of two squares if and only if $v_p(n)$ is even for every prime $p$ congruent to 3 modulo 4) or \cite[Example 2.2]{els} show that one must be careful with this philosophy.} should not have it.
A theorem of Wirsing \cite{wirsing} actually asserts that almost all sets
are totally primitive, where ``almost all'' refers to the construction of a random
subset of $\N$ by selecting each integer into the set with probability
$1/2$ independently of each other.

On the other hand, Ruzsa recently showed \cite{ruzsa} in the Number Theory Web Seminar that the widely believed Hardy-Littlewood
prime tuples conjecture implies that the signed set of primes
$\Pp\cup(-\Pp)$ is aymptotically additively $\Z$-decomposable, i.e. there exist sets $\mathcal{A},\mathcal{B}\subset \Z$ such that $\Pp\cup(-\Pp)\sim\mathcal{A}-\mathcal{B}$.
Our first result shows that this property is actually typical.
To state it,
let $(\xi_n)_{n\in\N}$ be a sequence of independent identically distributed Bernoulli variables satisfying $\Pr(\xi_n=1)=\Pr(\xi_n=0)=1/2$ for each $n\in\N$.
\begin{theorem}
Almost all symmetric subsets of $\Z$ are additively $\Z$-decomposable.
More precisely, let
 $\mathcal{D}=\{n\in\Z : \xi_{\abs{n}=1}\}$. Then $\Pr(\mathcal{D}\text{ is $\Z$-decomposable})=1$.
\end{theorem}
This theorem will follow from a finite tuples property that we will prove. Regarding subsets of $\N$, the same finite tuple
property yields the following.
\begin{theorem}
\label{cheapMrr}
Almost all subsets of $\N$ contain a sumset $\mathcal{A}+\mathcal{B}$ where both summands are infinite.
More precisely, let
$\mathcal{C}=\{n\in\N : \xi_n=1\}$. Then with probability 1,
$\mathcal{C}$ contains
 a sumset $\mathcal{A}+\mathcal{B}$ where both summands are infinite.
%=\mathcal{C}\cup -\mathcal{C}
\end{theorem}
A theorem of Granville \cite{granville} shows that the Hardy-Littlewood conjecture implies that set of primes contains  a sumset $\mathcal{A}+\mathcal{B}$ where both summands are infinite. 
Thus what is known of the primes, resp. the signed primes, under the Hardy-Littlewood conjecture, is true of almost every subset of $\N$, resp. almost every symmetric subset of  $\Z$.
Further, let us observe that a much stronger statement than Theorem \ref{cheapMrr}
holds, namely that every dense subset of $\N$ contains a sumset $\mathcal{A}+\mathcal{B}$ where both summands are infinite; this was conjectured by Erd\H{o}s and proven by Moreira, Richter and Robertson \cite{mrr} (see also \cite{hostSumset} for a simpler proof). Almost every subset being dense, this statement implies Theorem \ref{cheapMrr}, which may be thought of as a ``cheap'' version
of the theorem of  Moreira, Richter and Robertson.

Regarding decomposability of subsets of $\N$, we consider new probability 
distributions whose mass is concentrated ``near'' a fixed set of interest such as the set of the primes.
First we introduce a standard notational convention: a set (or equivalently an increasing sequence) of integers is denoted by a calligraphic letter (e.g. $\mathcal{A}$), its elements are denoted by the corresponding
lower case letter (e.g. $\mathcal{A}=\{a_n : n\in\N\}$), its counting function by the corresponding upper case letter (e.g. $A(x)=\abs{\mathcal{A}\cap [1,x]}$).

Let $\mathcal{S}$ be an infinite set and define a function $f=f_\mathcal{S}$ by $f(x)=x/S(x)$. 
We make the following two hypotheses.
\begin{enumerate}[label={\bfseries S\arabic*}]
\item $f(x)$ tends to infinity as $x$ does.
\item The number of $s_k\leq x$ satisfying $s_{k+2}-s_k\leq h$ is $o_h(S(x)/\log f(x))$.
\end{enumerate}
%Let $p_n$ denote the the $n$th prime.
Then we fix a sequence $(\delta_n)_{n\in\N}$ 
which has the following properties.
\begin{enumerate}[label={\bfseries D\arabic*}]
\item $\delta_n \geq 2$ for all but $o(x)$ integers $n\leq x$.
\item There exists an integer $\ell$ such that $\sum_{s_k\leq x}\prod_{i=0}^{\ell-1}\delta_{k+i}^{-1}=o(S(x)/\log f(x))$
\item $1\leq \delta_n\leq (s_{n+1}-s_{n-1})/2$ for all $n\geq 2$.
\end{enumerate}
An obvious consequence of {\bfseries D2} is that $\limsup_n \delta_n=\infty$.
Also we may assume that $\ell$ is even.
We will often denote $\delta_n^{-1}$ by $\eta_n$.
Moreover we consider a sequence $\varepsilon_n$ of independent random integers such that 
%$\varepsilon_1=\varepsilon_2=\varepsilon_3=0$
%and for $n\geq 4$,
\begin{enumerate}[label={\bfseries E\arabic*}]
\item $-(s_n-s_{n-1})/2< \varepsilon_n\leq (s_{n+1}-s_n)/2$ for all $n\in\N$.
\item For every $k\in\Z$, we have $\Pr(\varepsilon_n=k)\leq\delta_n^{-1}$.
\end{enumerate}
Thanks to {\bfseries D3}, such a sequence exists: we may take $\varepsilon_n$ to be uniformly distributed on the interval of integers $(-(s_n-s_{n-1})/2, (s_{n+1}-s_n)/2]$ for all $n\in\N$, for instance.
Note that {\bfseries E2} implies that 
%at least 
%$2\delta_n-1$ integers
%are reached by $\varepsilon_n$ 
%with probability at least $1/(2\delta_n)$ so that 
$\Var(\varepsilon_n)\gg \delta_n^2$, so we require a certain amount of dispersion; otherwise the problem is too close to the deterministic question of whether $\mathcal{S}$ itself is asymptotically additively decomposable, which is completely different.
%It will be convenient to assume that $\delta_n \geq 2$ for all $n\geq 4$.

Finally we consider the random sequence $c_n=\varepsilon_n+s_n$ and the random set $\mathcal{C}_n=\{c_n : n\in\N\}\subset\N$.
Observe that the definition of $\varepsilon_n$ ensures that $c_n<c_{n+1}$ for all $n$,
so that the sequence $\varepsilon_n$ uniquely determines $\mathcal{C}$.

\begin{theorem}
%Let $(\delta_n)_{n\in\N}$ be an increasing sequence of real numbers that tend to infinity arbitrarily slowly.
%Let $(\varepsilon_n)_{n\in\N}$ be a sequence of pairwise independent random variables, where $\varepsilon_n$ is uniformly distributed over $[-\delta_n,\delta_n]$ for each $n\in\N$. 
%Consider $\mathcal{C}=\{p_n+\varepsilon_n\}$ where $(p_n)_{n\N}=(2,3,5,\ldots)$ is the sequence of the primes, and
%$D=\{s_n+\varepsilon_n\}$ where $s_n$ is the sequence of sums of two squares.
The set $\mathcal{C}$ is almost surely totally primitive.
\end{theorem}
The following theorem shows that we can take $\mathcal{S}$ to be the
set of the primes or the set of the sums of two squares.
Basically, it suffices to disturb the $n$th prime or sum of two squares by a random
integer of standard deviation a small power of $\log\log n$.
\begin{theorem}
The set $\mathcal{P}$ of the primes and the set $2\mathcal{Q}$ of the sums of two squares fulfill the hypotheses {\bfseries S1,S2}.
When $\mathcal{S}$ is either of these two sets and the sequence $\delta$ satisfies $\delta_n\leq (s_{n+1}-s_{n-1})/2$ for all $n\in\N$ and
$\delta_n\gg \min((s_{n+1}-s_{n-1})/2,(\log\log n)^\iota )$ for $n\geq 2$  and
some $\iota >0$ arbitrarily small, the properties {\bfseries D1-3} hold.
\end{theorem}
In this sense, almost every small perturbation of either of these two sets is totally primitive.
Regarding sums of two squares, this statement is a kind of inverse Atkin theorem: Atkin \cite{atkin} proved that for almost every small perturbation
$\mathcal{Q'}$ of the set $\mathcal{Q}$ of the squares, the sumset $2\mathcal{Q'}$ is dense, in sharp contrast with $2\mathcal{Q}$.
Here we perturb the sumset instead of perturbing the summands, and lose almost surely the decomposability property.

\section{Proof of Theorems 1 and 2}
Both results rely on the \textit{finite tuples property}, which is the property
 of a set $\mathcal{C}\subset\Z$ such that for any
finite $\mathcal{H}\subset \Z$, there exist infinitely many $n\in \N$ such that $n+\mathcal{H}\subset \C$.

Here let $(\xi_n)_{n\in\N}$ be a sequence of independent identically distributed Bernoulli variables satisfying $\Pr(\xi_n=1)=\Pr(\xi_n=0)=1/2$ for each $n\in\N$ and let 
$\mathcal{C}=\{n\in\N : \xi_n=1\}$.

\begin{lm}
Almost all subsets of $\N$, and therefore almost all symmetric subsets of $\Z$, have the finite tuples property. That is, with probability 1, $\C$
and therefore $\C\cup-\C$ have the finite tuples property.
\end{lm}
\begin{proof}
Fix a particular finite non empty set $\Hh\subset \Z$ of cardinality $k$.
Then for any given $n\in\N$ large enough, the probability that a random subset $\C\subset\N$ is such that $n+\Hh\subset \C$ is $2^{-k}$. One can extract an infinite sequence $n_j$ of integers such that the events
$n_j+\Hh\subset C$ are pairwise independent, for instance $n_j=j(\diam\Hh +1)$ where $\diam$ denotes the diameter (difference between maximum and minimum). 
Therefore, by the Borel-Cantelli lemma, with probability 1, there exist infinitely many $n$ such that $n+\Hh\subset \C$. 
Now this is true for any particular $\Hh$, but since there are only countably many 
finite tuples
$\Hh\subset \N$, we can take the intersection and conclude.
\end{proof}
\begin{lm}
Any symmetric subset of $\Z$ satisfying the finite tuples property is additively $\Z$-decomposable; it is in fact the difference set of two infinite subsets of $\N$.
\end{lm}
The Lemmata 5 and 6 immediately imply Theorem 1.
\begin{proof}
Indeed, let $\D$ be a symmetric subset of $\Z$ satisying the finite tuple property (and therefore infinite), and let $\D=\{d_1,d_2,\dots\}$ be an ordering of $\D$. 
We will construct iteratively increasing sequences $a_k,b_k$ of positive integers
such that $\D=\{a_k\}-\{b_k\}$ and $d_k=a_k-b_k$.
To achieve this, start with any pair $(a_1,b_1)$ such that $d_1=a_1-b_1$.
Assuming finite increasing sequences of positive integers $a_1,\ldots,a_k$ (forming a set $\A_k$) and $b_1,\ldots,b_k$ (forming a set $\B_k$) have already been constructed and satisfy $a_i-b_i=d_i$ and $\A_k-\B_k\subset \D$,
let us construct $a_{k+1}\notin\A_k$ and $b_{k+1}\notin\B_k$
such that $\A_k\cup\{a_{k+1}\}-\B_k\cup\{b_{k+1}\}\subset \D$. Let us look for a positive integer $x$ such that $x-\B_k\subset \D$
and $x-d_{k+1}-\A_k\subset \D$ (by symmetry equivalently 
$-x+d_{k+1}+\A_k\subset \D$).
There exist infinitely many such $x$, due to the finite tuples property applied to the tuple $-\B_k\cup -(d_{k+1}+\A_k)$.
So there exists such an $x$ outside of the finite set $\A_k\cup (d_{k+1}+\B_k)$, and we pick $a_{k+1}=x$ and
$b_{k+1}=x-d_{k+1}$ and we are done.
\end{proof}

\begin{lm}
Any subset $\D$ of $\N$ satisfying the finite tuples property contains a sumset $\A+\B$ where both summands are infinite.
\end{lm}
Lemmata 5 and 7 imply Theorem 2.
\begin{proof}
We will construct iteratively increasing sequences $a_k,b_k$ of positive integers
such that $\{a_k : k\in\N\}+\{b_k : k\in\N\}\subset \D$.
We start with any pair $(a_1,b_1)$ such that $a_1+b_1\in \D$.
Assuming pairwise distinct $a_1,\ldots,a_k$ (forming a set $\A_k$) and $b_1,\ldots,b_k$ (forming a set $\B_k$) have already been constructed and satisfy $\A_k+\B_k\subset \D$, we first select $a_{k+1}$ to be an
integer outside $A_k$ satisfying $a_{k+1}+B_k\subset \D$, which exists by the finite tuple property.
We then set $\A_{k+1}=\{a_{k+1}\}\cup \A_k$ and take $b_{k+1}\in\N$ outside $\B_k$ such that $b_{k+1}+\A_{k+1}\subset \D$.
\end{proof}
\section{Proof of Theorems 3 and 4}

\begin{proof}[Proof of Theorem 3]
Let $\Dec$ be the set of additively decomposable subsets of $\N$.
% of sets $\mathcal{C}$ of the form $\mathcal{C}=\mathcal{A}+\mathcal{B}$ for some $\mathcal{A},\mathcal{B}\subset\N$ satisfying $ \min(\abs{\mathcal{A}},\abs{\mathcal{B}})\geq 2$.
Since an asymptotically additively decomposable set differs from an element of $\Dec$ by finitely many editions, of which there are countably many, it suffices to show that
$\Pr(\mathcal{C}\in \Dec)=0$.
Let $\Dec_1$ be the  set of sets $\mathcal{C}$ of the form $\mathcal{C}=\mathcal{A}+\mathcal{B}$ for some $\mathcal{A},\mathcal{B}\subset\N$ satisfying $ \min(\abs{\mathcal{A}},\abs{\mathcal{B}})\geq 2$ and
$A(x)+B(x)<\frac{\log 2}{2}\frac{x}{f(x)\log f(x)}$ for infinitely many integers $x$.
Let $\Dec_2$ be the  set of sets $\mathcal{C}$ of the form $\mathcal{C}=\mathcal{A}+\mathcal{B}$ for some $\mathcal{A},\mathcal{B}\subset\N$ satisfying $ 2\leq \min(\abs{\mathcal{A}},\abs{\mathcal{B}})<\infty$.
 Let $\Dec_3=\Dec\setminus (\Dec_1\cup \Dec_2)$ so that $\Dec=\Dec_1\cup\Dec_2\cup \Dec_3$.

First we show that 
$\Pr(\mathcal{C}\in\Dec_1)=0.$
We note that $\Dec_1=\bigcap_{x_0\geq 1}\bigcup_{x\geq x_0}\Dec_1(x)$
where 
$\Dec_1(x)$ is the  set of sets $\mathcal{C}$ of the form $\mathcal{C}=\mathcal{A}+\mathcal{B}$ for some $\mathcal{A},\mathcal{B}\subset\N$ satisfying $ \min(\abs{\mathcal{A}},\abs{\mathcal{B}})\geq 2$ and
$A(x)+B(x)<\frac{\log 2}{2}\frac{x}{f(x)\log f(x)}$.
% where for a set $\mathcal{D}\subset \N$, we let $D(x)=\abs{D\cap [0,x]}.$
Let $x$ be a large integer and $r\leq x/2$.
We note that if $\mathcal{C}=\mathcal{A}+\mathcal{B}$, where $A(x)+B(x)<r$, then $\mathcal{C}\cap [0,x]$ is
one of at most $r\binom{2x}{r}\leq r(2ex/r)^r$ sets.
But for any set in $\mathcal{D}\subset [0,x]$, the probability that $\mathcal{C}\cap [0,x]=\mathcal{D}$ is at most
$\prod_{n : s_{n+1}<x}\delta_n^{-1}\leq 2^{-(1+o(1))x/f(x)}$
by definition of $\mathcal{S}$ and property {\bfseries D1}. Taking $r=\lfloor\frac{\log 2}{2}\frac{x}{f(x)\log f(x)}\rfloor$, we infer
that 
 $\Pr(\mathcal{C}\in\Dec_1(x))=\exp(-\Omega(x/f(x))=o(1)$.
This implies that $\Pr(\mathcal{C}\in\Dec_1)=0.$

Now we seek to show that $\Pr(\mathcal{C}\in \Dec_2)=0.$
Note that every set $\mathcal{D}\in \Dec_2$ satisfies $\limsup \min(d_{k+1}-d_k,d_k-d_{k-1})<\infty$. Indeed, if $\mathcal{D}=\mathcal{A}+\mathcal{B}$ where $\mathcal{B}$ is finite, let $H$ be the largest gap between two consecutive elements of $\mathcal{B}$; then $\min(d_{k+1}-d_k,d_k-d_{k-1})\leq H$ for every $k$.
Further we claim that for any $h$ and $k$ integers,
\begin{equation}
\label{ckckp1}
\Pr(c_k+h=c_{k+1})\leq h^2\eta_k\eta_{k+1}.
\end{equation}
To see this, note that $c_k+h=c_{k+1}$ implies that $\varepsilon_k\in [(s_{k+1}-s_k)/2-h,(s_{k+1}-s_k)/2)$
and similarly $\varepsilon_{k+1}\in [-(s_{k+1}-s_k)/2,-(s_{k+1}-s_k)/2+h]$.
We infer that 
$\Pr(c_{k+1}-c_k\leq H)\leq H^3\eta_k\eta_{k+1}$.
Therefore $$\Pr(\min(c_{k+1}-c_k,c_k-c_{k-1})\leq H)\leq H^3\eta_k(\eta_{k+1}+\eta_{k-1}).$$
Since $\inf\eta_k=0$, it follows that 
$$\Pr(\forall k\,\min(c_{k+1}-c_k,c_k-c_{k-1})\leq H)\leq 
\inf_k \Pr(\min(c_{k+1}-c_k,c_k-c_{k-1})\leq H)=0
$$ 
hence $\Pr(\mathcal{C}\in\Dec_2)=0$.

We turn to $\Dec_3$.
We note that $\Dec_3=\bigcup_{x_0\geq 1}\bigcap_{x\geq x_0}\Dec_3(x)$
where 
$\Dec_3(x)$ is the  set of sets $\mathcal{C}$ of the form $\mathcal{C}=\mathcal{A}+\mathcal{B}$ for some infinite sets $\mathcal{A},\mathcal{B}\subset\N$ satisfying
$A(x)+B(x)\geq\frac{\log 2}{2}\frac{x}{f(x)\log f(x)}$.
Suppose $\mathcal{C}=\mathcal{A}+\mathcal{B}$ where $\mathcal{A}$ and $\mathcal{B}$ are infinite and $A(x)+B(x)\gg x/(f(x)\log f(x))$, in particular $\max(A(x),B(x))\geq \kappa x/(f(x)\log f(x))$ for some constant $\kappa>0$ and every $x$ large enough.
Let $x$ be a large integer and assume without loss of generality that
$A(x)\geq \kappa x/(f(x)\log f(x))$.
Let $b_1<b_2<\ldots <b_\ell$ be in $\mathcal{B}$.
Then we know that for any $n\in \mathcal{A}$ the set $\{n+b_i : i\in [\ell]\}$ lies in $\mathcal{C}$.
So the number of $n\in [0,x]$ such that  $\{n+b_i : i\in [\ell]\}\subset \mathcal{C}$ is at least $A(x)$.
Therefore 
\begin{equation}
\label{cdec3x}
\Pr(\mathcal{C}\in\Dec_3(x))\leq \Pr(\abs{[0,x]\cap \bigcap_{i\in [\ell]} (\mathcal{C}-b_i)}\geq\kappa x/(f(x)\log f(x))).
\end{equation}

Now
suppose $\{n+b_i : i\in [\ell]\}\subset \mathcal{C}$. 
Then
$n+b_i\in \mathcal{C}$ means $n+b_i=c_{j_i}=s_{j_i}+\varepsilon_{j_i}$ for some $j_1<j_2<\ldots<j_\ell$.
However the first property of $\varepsilon_n$ implies easily that either $j_i=j_1+i-1$ for every $i\in [\ell]$,
or $s_{j_i+2}-s_{j_i}<2(b_{i+1}-b_i)$ for some $i\in [\ell]$.
Recall that, by hypothesis, the number of $k$ such that $s_{k+1}\leq x$ and $s_{k+2}-s_k<2h$ is $o_h(x/(f(x)\log f(x))$.
%classical sieve estimates \cite[Theorem 5.7]{HR} implies, for any given integers $0<a<b$ , that the number of primes $p\leq x$
%such that $p+a$ and $p+b$ are primes is at most $O_{a,b}(x/\log^3x)$, so that the number
%of $p_k\leq x$ satisfying $p_{k+2}-p_k<2h$ is $O_h(x/\log^3x)$.

So in total
\begin{align*}
 \E[\abs{\{n\leq x : \{n+b_i : i\in [\ell]\}\subset \mathcal{C}\}}]
&= \sum_{n\leq x} \Pr(\{n+b_i : i\in [\ell]\}\subset \mathcal{C})\\
&\leq \sum_{k:s_{k+1}\leq x+b_\ell}1_{s_{k+2}-s_k<2h}\\
&+\sum_{k:s_{k+1}\leq x+b_1}\Pr(\forall i\in [\ell]\, : c_k+b_i-b_1=c_{k+i-1})
\end{align*}
Recall that $\Pr(c_k+h=c_{k+1})\leq h^2\eta_k\eta_{k+1}$ by equation \eqref{ckckp1}.
Assume $\ell=2\ell'$.
By independence,
\begin{align*}
\Pr(\forall i\in [\ell]\, : c_k+b_i-b_1=c_{k+i-1})&\leq
\Pr(\forall i\in [\ell']\, : c_{k+2i-2}+b_{2i}-b_{2i-1}=c_{k+2i-1})\\
&\ll \prod_{i=0}^{\ell-1}\eta_{k+i}
\end{align*}
Since $\sum_{s_k\leq x}\prod_{i=0}^{\ell-1}\eta_{k+i}=o(x/(f(x)\log f(x)))$,
we obtain that 
$$\E[\abs{\{n\leq x : \{n+b_i : i\in [\ell]\}\subset \mathcal{C}\}}]=o(x/(f(x)\log f(x))).$$
By Markov's inequality, 
we find that 
$$\Pr(\abs{\{n\leq x : \{n+b_i : i\in [\ell]\}\subset \mathcal{C}\}}\geq\kappa x/(f(x)\log f(x)))=o(1).$$
In view of equation \eqref{cdec3x} infer that $\Pr(\mathcal{C}\in\Dec_3(x))=o(1)$ whence $\Pr(\mathcal{C}\in\Dec_3)=0$.
\end{proof}

\begin{proof}[Proof of Theorem 4]
By the prime number theorem, when $\mathcal{S}$ is the set $\mathcal{P}$ of the primes, we have $f_\mathcal{P}(x)\sim \log x$ which proves {\bfseries S1}.
Further, let  $\pi_m(x)$ be the number of primes $p\leq x$ such that 
$p+m\in\mathcal{P}$. Then by Selberg's sieve $\pi_m(x)\ll \frac{x}{\log^2x}\prod_{p\mid m}(1+1/p)$, where the implied constant is absolute; see for instance \cite{HR}.
Then the number of $p_k\leq x$ such that $p_{k+1}\leq p_k+M$
is at most
\begin{equation}
\label{smallgap}
\sum_{m\leq M}\pi_m(x)\ll \frac{x}{\log^2x}
\sum_{m\leq M}\prod_{p\mid m}(1+1/p)=\frac{x}{\log^2x}\sum_{d\leq M}
\lfloor M/d\rfloor \frac{\mu(d)^2}{d}\ll M\frac{x}{\log^2x}.
\end{equation}
In particular
the number of primes $p_k\leq x$ such that $p_{k+2}-p_k\leq h$ is
$O_h(x/f(x)^2)=o_h(x/(f(x)\log f(x)))$ (in fact it is $O_h(x/f(x)^3)$).
Thus {\bfseries S2} holds.

Let the sequence $\delta$ satisfy $\delta_n\leq (s_{n+1}-s_{n-1})/2$ for all $n\in\N$ and
$\delta_n\gg \min((s_{n+1}-s_{n-1})/2,(\log\log n)^\iota )$ for $n\geq 2$  and
some $\iota >0$.
Let us check that this sequence satisfies the properties required above.
{\bfseries D1,3} are obvious.
Applying \eqref{smallgap} with $M= (\log\log x)^\iota $, we see that $\delta_n\gg(\log\log n)^\iota $ for all but $o(x/(f(x)\log f(x)))$ primes $p_n\leq x$.
As a result, {\bfseries D2} is satisfied for any $\ell >\iota^ {-1}$.

Similarly, for the set $\mathcal{S}=2\mathcal{Q}$ of sums of two squares, we have $f_\mathcal{S}(x)\asymp \sqrt{\log x}$ by a classical result of Landau, which proves {\bfseries S1}.
Further, let  $\theta_m(x)$ be the number of sums of two squares $s\leq x$ such that 
$s+m\in\mathcal{S}$. Then again $\theta_m(x)\ll \frac{x}{\log x}\prod_{p\mid m,p\equiv 3\mod 4}(1+1/p)$, where the implied constant is absolute; this may be achieved via Selberg's sieve, see \cite{nowak}.
 Arguing like in equation \eqref{smallgap}, we find that the number of $s_k\leq x$ such that $s_{k+1}\leq s_k+M$
is at most
 $M\frac{x}{\log x}.$
In particular
the number of sums of two squares $s_k\leq x$ such that $s_{k+2}-p_k\leq h$ is
$O_h(x/f(x)^2)=o_h(x/(f(x)\log f(x)))$ (in fact here too, it is $O_h(x/f(x)^3)$).
Thus {\bfseries S2} holds.

Simultaneously this proves \textbf{D1}.
{\bfseries D3} holds by definition and \textbf{D2} may be proven along the same lines as above.
\end{proof}
\section*{Acknowledgments}
This paper was supported by the joint FWF-ANR project Arithrand: FWF: I 4945-N and ANR-20-CE91-0006.
The author would like to thank Christian Elsholtz and Oleksiy Klurman for useful conversations.

\end{document}